\documentclass[11pt]{amsart}
 \usepackage[margin=1.5in]{geometry}
\usepackage{bm}

\usepackage{mathpazo}
\usepackage{latexsym}
\usepackage{amssymb}
\usepackage{amsmath}
\usepackage{amsthm}  
\usepackage{hyperref}

\numberwithin{equation}{subsection}
\newtheorem{theorem}{Theorem}[section]
\newtheorem{lemma}[theorem]{Lemma}
\newtheorem{proposition}[theorem]{Proposition}
\newtheorem{corollary}[theorem]{Corollary}

\theoremstyle{definition}
\newtheorem{remark}[theorem]{Remark}

\DeclareMathOperator{\Lie}{Lie}

\DeclareMathOperator{\gl}{GL}
\def\sp{{\mathop{\rm Sp}}}
\def\esix{{\rm E}_6}

\def\stable{{\rm st}}

\def\ball{{\mathbb B}}
\def\hp{{\mathbb H}}

\newcommand{\stk}[1]{{\mathcal #1}}
\newcommand{\tilstk}[1]{{\til{\stk #1}}}
\newcommand{\barstk}[1]{{\bar{\stk #1}}}
\newcommand{\coarse}[1]{{\underline{#1}}}

\usepackage{diagrams}
\newarrow{to}---->
\newarrow{dashedto}{}{dash}{}{dash}>
\newarrow{mto}|--->
\newarrow{impto}===={=>}
\newarrowtail{<=}\Rightarrow\Leftarrow\Uparrow\Downarrow
\newarrow{bboth}{<=}==={=>}
\newarrow{inject}{hooka}---{vee}
\newarrow{iline}{hooka}----
\newarrow{surject}----{>>}

\def\marked{{\rm m}}
\def\setcomp{\smallsetminus}

\def\aff{{\mathbb A}}
\def\cx{{\mathbb C}}
\def\ff{{\mathbb F}}
\def\proj{{\mathbb P}}
\def\rat{{\mathbb Q}}
\def\real{{\mathbb R}}
\def\integ{{\mathbb Z}}
\def\gp{{\mathbb G}}

\def\ww{{\mathbb W}}

\def\mmu{\bm{\mu}}
\def\idp{{\frak p}}
\def\idl{{\frak l}}

\def\ida{{\frak a}}
\def\idb{{\frak b}}
\def\mono{{\sf M}}

\def\calz{{\mathcal Z}}

\def\cald{{\mathcal D}}
\def\calf{{\mathcal F}}
\def\caln{{\mathcal N}}
\def\calo{{\mathcal O}}

\def\calx{{\mathcal X}}

\def\caly{{\mathcal Y}}

\DeclareMathOperator{\T}{T}
\DeclareMathOperator{\K}{K}
\DeclareMathOperator{\mat}{Mat}

\newcommand{\til}[1]{{\widetilde{#1}}}
\newcommand{\st}[1]{\{#1\}}
\newcommand{\ang}[1]{{{\langle #1 \rangle}}}
\newcommand{\hang}[1]{\langle\!\langle #1 \rangle\!\rangle}
\newcommand{\sra}[1]{\stackrel{#1}{\rightarrow}}
\newcommand{\rest}[1]{|_{#1}}

\def\liehodge{{\mathfrak{hg}}}
\def\lie{\mathfrak}

\global\let\det\undefined
\DeclareMathOperator{\det}{det}
\global\let\hom\undefined
\DeclareMathOperator{\hom}{Hom}
\DeclareMathOperator{\id}{id}
\DeclareMathOperator{\aut}{Aut}
\DeclareMathOperator{\gal}{Gal}
\DeclareMathOperator{\SL}{SL}
\DeclareMathOperator{\gu}{GU}
\DeclareMathOperator{\End}{End}
\DeclareMathOperator{\tr}{tr}

\DeclareMathOperator{\Frac}{Frac}
\def\res{\bm{{\rm R}}}
\DeclareMathOperator{\mt}{MT}
\DeclareMathOperator{\hg}{Hg}
\DeclareMathOperator{\cl}{Cl}

\def\gen{{\rm gen}}
\def\u{{\rm U}}
\DeclareMathOperator{\su}{SU}

\def\fin{{\rm fin}}
\DeclareMathOperator{\Cl}{Cl}

\def\ra{\rightarrow}
\def\tensor{\otimes}
\def\iso{\cong}
\def\cross{\times}
\def\inject{\hookrightarrow}
\def\bs{\backslash}
\def\sriso{\stackrel{\sim}{\rightarrow}}

\def\units{^\cross}
\DeclareMathOperator{\spec}{Spec}

\def\inv{^{-1}}
\def\twiddle{\sim}

\def\ab{^{{\rm ab}}}

\newcommand{\powser}[1]{[\![#1]\!]}
\newcommand{\laurser}[1]{(\!(#1)\!)}

\newenvironment{alphabetize}{\begin{enumerate}

}{\end{enumerate}}

\begin{document}

\title{On the abelian fivefolds attached to cubic surfaces}

\author{Jeffrey D. Achter}
\email{j.achter@colostate.edu}
\address{Department of Mathematics, Colorado State University, Fort
Collins, CO 80523} 
\urladdr{http://www.math.colostate.edu/~achter}

\thanks{This work was partially supported by a grant from the Simons
  foundation (204164).}
\subjclass[2010]{Primary 14J10; Secondary 11G18, 14D05, 14K30}

\begin{abstract}
To a family of smooth projective cubic surfaces one can canonically
associate a family of abelian fivefolds.  In characteristic zero, we
calculate the Hodge groups of the abelian varieties which arise
in this way.  In arbitrary characteristic we calculate the monodromy
group of the universal family of abelian varieties, and thus show that 
the Galois group of the 27 lines on a general cubic surface in
positive characteristic is as large as possible.
\end{abstract}

\maketitle

\section{Introduction}

The moduli space of smooth
complex cubic surfaces is open in an arithmetic quotient of the
complex $4$-ball \cite{act02,dolgachevetal05}.  This statement
is the complex specialization of a homeomorphism $\tau$ of stacks over
$\integ[\zeta_3,1/6]$ between $\stk S_\stable$, the moduli space of
stable cubic surfaces, and $\stk M$, the moduli space of principally
polarized abelian fivefolds equipped with an action by
$\integ[\zeta_3]$ of signature $(4,1)$ \cite{achtercubmod}. In
particular, to a cubic surface $Y$ one may canonically associate an
abelian fivefold $X = Q(Y)$; it is natural to try and characterize the
abelian varieties which arise in this way.

Different strategies have the potential to yield different
information.  For example, it is possible \cite[Cor.\ 5.9]{achtercubmod} to
reconstruct $Y$ from the theta divisor of the polarized abelian
variety $X$.  However, this geometric method doesn't obviously allow
one to understand, say, the endomorphism ring $\End(X)$.
Alternatively, at least over the complex numbers, one can (as Carlson
and Toledo do \cite{carlsontoledo13}) explicitly evaluate the 
necessary period integrals. This approach has the virtue of explicitly
relating symmetries of $Y$ to endomorphisms of $X$, but seems
unlikely to yield an explicit classification of all possible abelian varieties.

The main theme of the present paper is that, via $\tau$, techniques
from the theory of Shimura varieties (and their compactifications) can
be brought to bear on the moduli space of cubic surfaces.  

In Section \ref{secmt}, we completely classify the Hodge groups of
complex abelian fivefolds of the form $Q(Y)$. Proposition \ref{prophg}
describes the endomorphism algebras, and thus the
Hodge groups, which could conceivably arise in this way, while
Proposition \ref{propshape} guarantees that such behavior really does
occur.

In Section \ref{secmono}, we calculate the monodromy group of the
universal abelian variety when pulled back to $\stk S$.  In fact, the
Torelli map $\tau$ passes through the moduli stack $\stk T$ of cubic
threefolds, and we also calculate the monodromy group of the universal
abelian variety over $\stk T$ (Theorem \ref{thmono}).  As a
consequence, we are able to show (Corollary \ref{corexistsmallend})
that there are cubic surfaces $Y$ defined over $\rat$ with
endomorphism ring exactly equal to $\integ[\zeta_3]$.  The paper
concludes by using the theory of arithmetic toroidal compactifications
of Shimura varieties of PEL type in order to compute the Galois group
of the 27 lines on a sufficiently general cubic surface in 
characteristic at least five (Proposition \ref{propfieldoflines}).

The classification in Section \ref{secmt} and the
monodromy calculation in Section \ref{secmono} require, respectively,
a method for constructing principally polarized abelian varieties with
specified endomorphism ring and information about unitary Shimura
varieties.  These are developed in a broader context than that
required for understanding cubic surfaces.  This generality
imposes no extra burden on the proof, and these results may be of
independent interest.

Given a polarized abelian variety $(X,\lambda)$ and an
order $\calo \subset \End(X)$, by adjusting $(X,\lambda)$ in its
isogeny class one can either find a principal polarization, or expand
$\calo$ to a maximal order; but achieving both simultaneously seems
delicate.  
Section \ref{subsecendring}, especially Lemma \ref{lemcmsig}, gives
conditions under which one can find a principally polarized abelian
variety with specified commutative endomorphism algebra and an action by the ring
of integers of a quadratic imaginary field of specified signature.

The mod-$\ell$ monodromy calculation for $\stk S$ comes down to a
question of the irreducible components of $\stk M^\ell$, the moduli
space of $\integ[\zeta_3]$-abelian varieties equipped with a principal
level $\ell$ structure.  This is worked out in detail in Lemma
\ref{lemunitarymono} for an arbitrary PEL Shimura variety attached to
$\integ[\zeta_3]$; the proof relies both on class field theory and on
arithmetic compactifications.

Part of the material in Section \ref{secmt} was developed in response
to \cite{carlsontoledo13} and subsequent exchanges with its authors;
it's a pleasure to thank J.\ Carlson and D.\ Toledo for their insights
and encouragement.  C.\ Hall and N.\ Katz kindly noted a mistake in an
earlier work (see Remark \ref{remwhoops}), which is now corrected in
Section \ref{subsecunitarymono}.  B.\ Gross pointed out that
\cite{grosssigns} contains a simple criterion for 
condition \eqref{thecond}.  Finally, this paper benefited from the
referee's careful reading and concomitant suggestions.

\section{Background and Notation}

The following notation will be in force throughout the paper, {\em
  except} that: in Section \ref{subsecendring}, $E$ denotes an arbitrary
quadratic imaginary field, and not necessarily $\rat(\zeta_3)$; and in Section
\ref{subsecunitarymono}, $\stk M$ denotes the moduli space of abelian
varieties with $\integ[\zeta_3]$-action of arbitrary, fixed signature
$(r,s)$, and not necessarily signature $(4,1)$.

\subsection{Moduli spaces}

Let $\calo_E = \integ[\zeta_3]$ be the ring of Eisenstein integers,
with fraction field $E = \rat(\zeta_3)$.

The present work explores some consequences of the existence of
nontrivial maps between moduli spaces of cubic surfaces, cubic
threefolds, and abelian fivefolds.

Let $\stk S$ be the moduli stack of smooth cubic surfaces.  It is the
quotient of $\tilstk S$, the moduli scheme of smooth cubic forms in
four variables, and is partially compactified by $\stk S_\stable$, the
moduli stack of stable cubic surfaces.  Similarly, let $\stk T$ be the
moduli stack of smooth cubic threefolds, and let $\tilstk T$ be the
moduli scheme of smooth cubic forms in five variables.

Finally, let $\stk A_5$ be the moduli stack of principally polarized
abelian fivefolds, and let $\stk M$ be the moduli stack of principally polarized abelian
  fivefolds equipped with an action by $\integ[\zeta_3]$ of signature
  $(4,1)$.

The main work of \cite{achtercubmod} is to define and characterize
the morphisms $\til\varpi$ (over $\integ[1/2]$) and $\tau$ (over
$\integ[\zeta_3,1/6]$) indicated below: 
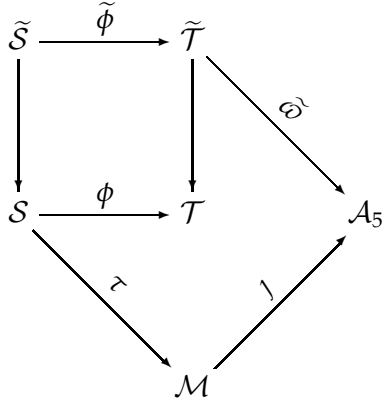
\begin{diagram}
\tilstk S & \rto^{\til \phi} & \tilstk T \\
\dto && \dto &\rdto>{\til\varpi}\\
\stk S & \rto^\phi & \stk T &&\stk A_5 \\
&\rdto>\tau&&\ruto^\jmath &\\
&&\stk M
\end{diagram}
In this diagram $\phi$ is the functor which, to a cubic surface $Y\subset
\proj^3$, associates the cubic threefold $Z\subset \proj^4$ which is
the cyclic triple cover of $\proj^3$ ramified exactly along $Y$.  The
morphism $\til\varpi$ comes from an arithmetic Prym construction; over
the complex numbers, it coincides with the intermediate Jacobian
functor.  (It seems plausible that $\til\varpi$ factors through $\stk
T$, but the author has not checked this.)

Let $\stk D \subset \stk M$ be the substack parametrizing
$\integ[\zeta_3]$-abelian schemes which are isomorphic to the product
of an elliptic curve and an abelian fourfold.  It is a horizontal
divisor; let $\stk N = \stk M - \stk D$.  The main result of
\cite{achtercubmod} is:

\begin{theorem}[{\cite[Thm.\ 5.7]{achtercubmod}}]
\label{thmoduli}
There is an isomorphism of stacks $\tau:\stk S \ra \stk N$ over
$\integ[\zeta_3,1/6]$ which
extends to a homeomorphism of stacks $\stk S_\stable \ra \stk M$.
\end{theorem}
(On the boundary, the existence of nodal cubic surfaces without
nontrivial automorphisms means that, even over $\cx$, $(\stk S_\stable\setcomp
\stk S)_\cx \ra \stk D_\cx$ is not an isomorphism of stacks; see
\cite[3.18]{act02} for details.)

The universal abelian fivefold, cubic surface and cubic threefold will
be respectively denoted by $f:\calx \ra \stk A_5$, $g:\caly \ra \stk
S$ and $h:\calz \ra \stk T$; we will abuse notation and use the same
letters for the universal objects over $\tilstk S$ and $\tilstk T$.
The middle cohomology of $\calz$ is a Tate twist of the first
cohomology of $\til\varpi^* \calx$; see \cite[Prop.\
3.6]{achtercubmod} for a precise statement.

As a shorthand for the Prym functor, for $t\in \tilstk T$ let
$P(\calz_t) = \calx_{\til\varpi(t)}$.  Thus, if $Z$ is a cubic
threefold, $P(Z)$ is the principally polarized abelian fivefold
attached to $Z$ by $\til\varpi$.  Similarly, let $F(Y)$ be the cubic
threefold which is the cyclic triple cover of $\proj^3$ ramified along
$Y$;  $F(\caly_s) =\calz_{\phi(s)}$.  Finally, let $Q(Y) = P(F(Y))$.

If $X$ is an abelian variety, denote its endomorphism algebra by
$\End(X)_\rat = \End(X)\tensor\rat$.

\subsection{Unitary groups}
\label{subsecunitary}

Let $V = \calo_E^{\oplus 5}$, and equip it with the Hermitian pairing
\[
\hang{x,y} = \sum_{1 \le j \le 4}x_j\bar y_j - x_5\bar y_5.
\]
  Then
$\ang{x,y} = \tr(\hang{x,y}/\sqrt{-3})$
is a symplectic, $\integ$-linear pairing on the underlying
$\integ$-module of $V$, and $\ang{ax,y}= \ang{x, \bar a y}$ for $a \in
\calo_E$.  
Define various
unitary group schemes over $\integ$ as follows:
\begin{align}
\gu &= \gu(V,\hang{\cdot,\cdot}) \notag \\
\u &= \u(V,\hang{\cdot,\cdot}) \notag \\
\su &= \su(V,\hang{\cdot,\cdot})\notag \\
\intertext{and the tori}
\T &= \res_{\calo_E,\integ}\gp_{m,E} \label{eqdeft}\\
\T_1 &= \res_{\calo_E,\integ}(\ker \caln_{\calo_E/\integ}) \label{eqdeftone}
\end{align}
where $\res_{\calo_E/\integ}$ is the restriction of scalars functor.
Then $\su$ is the derived group of $\gu$ and of $\u$, and there are the
following sequences of group schemes, exact as sheaves on
$(\spec \integ[1/6])^{{\rm et}}$:
\begin{diagram}[LaTeXeqno]
\label{diagesunitary} 1 & \rto & \su & \rto & \gu & \rto^\nu & \T & \rto & 1\\
1 & \rto &\su & \rto & \u & \rto^\nu & \T_1 & \rto & 1 
\end{diagram}
There is a natural inclusion of group schemes $\res_{\calo_E/\integ}\mmu_6 \inject
\T_1$. Let
\begin{equation}
\label{eqdefustar}
\u^* = \u \cross_{\T_1} \mmu_6,
\end{equation}
an extension of $\mmu_6$ by $\su$.

For each natural number $N$, $\su(\integ) \ra \su(\integ/N)$ is
surjective (Lemma \ref{lemimageunitary}).  Since $\T_1(\integ)=
\res_{\calo_E/\integ}\mmu_6(\integ)$ is the group of sixth roots of
unity (and not larger), $\u^*(\integ/N)$ is the image of the reduction map $\u(\integ) \ra
\u(\integ/N)$.

\section{Mumford-Tate groups and endomorphism rings}
\label{secmt}

Carlson and Toledo have raised the question of classifying the
Mumford-Tate groups of abelian varieties
associated to complex cubic surfaces.  In this section, we avail
ourselves of known results on endomorphisms of abelian varieties in
order to classify the algebras $\End(Q(Y))_\rat$ as $Y$ ranges over
all smooth complex cubic surfaces.  Briefly, we find that each abelian
variety has at most one simple factor of dimension greater than one;
this yields a coarse classification by ``shape'', i.e., by the
dimension of the largest simple factor and the dimension of its
endomorphism algebra. After an apparent digression
on the problem of constructing principally polarized abelian varieties
with large endomorphism ring (as opposed to endomorphism algebra), we
are able to show that every possible shape occurs in $\stk S$.

Carlson and Toledo have independently
investigated the Mumford-Tate groups of the abelian varieties which
arise here, and the reader is invited to consult
\cite{carlsontoledo13} for a complementary perspective which
emphasizes period integrals.  

\subsection{Topological monodromy}

Consider the spaces of complex cubic surfaces and threefolds.  By pulling back
the middle cohomology of the universal family of cubic threefolds, one
obtains a local system $\til\phi^*R^3h_*\integ$ on $\til S_\cx$ of polarized
$\integ[\zeta_3]$-modules of rank $5$ equipped with a Hermitian
pairing.  Fix a base point $s\in \til S(\cx)$; the local system
$\til\phi^*R^3h_*\integ$ is determined by a representation $\pi_1^\top(\til
S_\cx, s) \ra \aut(H^3(\stk Z_{\phi(s)},\integ))$ of the topological
fundamental group of $\til S_\cx$, whose image is simply
called the (topological) monodromy group of $\stk Y_\cx \ra \tilstk S_\cx$.

\begin{proposition}[{\cite[Thm.\ 2.14]{act02}}]\label{propactmono}
The monodromy group of the universal family of cubic surfaces is a
subgroup $\u^+(\integ)\subset \u(\integ)$ of index two.
\end{proposition}

In fact, \cite{act02} gives a beautifully explicit characterization of
$\u^+(\integ)$.  For present purposes, it suffices to note that since
$\su(\integ)$ has no subgroup of index two, $\u^+(\integ)$ must
contain $\su(\integ)$ and is therefore an extension of $\mmu_3$ by
$\su(\integ)$.
This will be used below (Section \ref{subsecmt}) in describing
Mumford-Tate groups of intermediate Jacobians of cubic threefolds
attached to cubic surfaces.  

The comparison theorem between Betti and \'etale cohomology lets one
conclude that the connnected component of the geometric
$\bar\rat_\ell$-monodromy group of $\stk Y \ra \tilstk S$ is
$\su_{\bar\rat_\ell}$.  However, a direct calculation (Theorem 
\ref{thmono}) below yields more
detailed information (see Remark \ref{remmonoqlbar}) about the
monodromy group of the \'etale cohomology.

\subsection{Hodge groups and Mumford-Tate groups}
\label{subsecmt}

Let $X/\cx$ be a complex abelian variety.  Then $W := H^1(X(\cx),\rat)$ is
a polarizable Hodge structure.  If the Hodge structure on $W$ is
realized as a map of $\real$-groups $h:\res_{\cx/\real}\gp_m \ra
\aut(W\tensor_\rat\real)$, then the Mumford-Tate group $\mt(W)$ is the
smallest algebraic subgroup $M/\rat$ such that $h$ factors through
$M_\real$.  The Hodge structure is called special if $\mt(W)$ is a torus.

In a similar sense, the Hodge group $\hg(W)$ of $W$ is the
algebraic hull of the image of the restriction of $h$ to the unit
circle in $\cx$.  From this description, it is clear that $\mt(W)$ is
an extension of $\gp_m$ by $\hg(W)$.  Let $\liehodge(W)$ be the Lie
algebra of $\hg(W)$.  Finally, let $\mt(X)$, $\hg(X)$ and $\liehodge(X)$
denote the corresponding objects for $W$. 

Now consider $\ww = \til\phi^*R^3h_*\rat$ as a polarized variation of
Hodge structure on $\tilstk S_\cx$.  There is a generic Mumford-Tate
group $\mt^\gen(\ww)$ associated to $\ww$.  It is characterized by the
fact that for each $s\in \tilstk S(\cx)$, $\mt(\ww_s)\subseteq
\mt^\gen(\ww)$; and for $s$ in the complement of a certain countable
union of proper ({\em a priori} analytic, but actually algebraic)
subvarieties of $\tilstk S$, there is equality $\mt(\ww_s)\iso
\mt^\gen(\ww)$.

\begin{lemma}
\label{lemgenmt}
The generic Mumford-Tate group of $\ww$ is  $\mt^\gen(\ww) =
\gu_\rat$.
\end{lemma}

\begin{proof}
Since $\stk M$ is a Shimura variety, special points are analytically
dense in $\stk M_\cx$, and thus in $\tilstk S_\cx$.  Consequently
\cite[Prop.\ 2]{andrefixedpart}, the derived  
subgroup of $\mt^\gen(\ww)$ is equal to the connected component of the
(Zariski closure of) the topological monodromy group.  The result now
follows from Proposition \ref{propactmono}.
\end{proof}

In order to classify all groups which arise as Mumford-Tate groups of
abelian varieties associated to cubic surfaces, in principle one could
work out all sub-Shimura data of the (connected) Shimura datum $(\u, \ball^4)$.  Instead,
we proceed as follows.  The Mumford-Tate group of an abelian variety is invariant
under isogeny of the abelian variety.  The Hecke orbit of each $s\in
\stk M(\cx)$ is analytically dense.  In particular, each point in
$\stk D(\cx)$, the complement of the Torelli map $\stk S(\cx) \ra \stk
M(\cx)$, is isogenous to one in $\stk N(\cx) = (\stk M\setcomp\stk D)(\cx)$, and thus
every Mumford-Tate group which arises in $\stk M$ already arises in
the variation of Hodge structure $\ww \ra \tilstk S_\cx$.
It turns out that, for a complex abelian fivefold with action by
$\integ[\zeta_3]$, the endomorphism algebra completely determines the Mumford-Tate
group.  Consequently, we proceed by classifying endomorphism
algebras and then listing the associated Mumford-Tate (actually,
Hodge) groups.

\subsection{Mumford-Tate groups and endomorphism algebras}
\label{subsecendcx}

In this section we work up to isogeny, and in particular ``simple''
means ``simple up to isogeny''.

As a preview, note that if one is willing to restrict attention to
abelian fivefolds with $\integ[\zeta_3]$-action which are simple up to
isogeny, then the calculation is much shorter.  Indeed, suppose $X =
(X,\iota,\lambda) \in \stk M(\cx)$ is simple.  Then, using the fact
that $\dim(X) = 5$ is prime, one quickly sees that $\End(X)_\rat$ is
either $\rat(\zeta_3)$ or a CM field $K$ such that $[K:\rat] = 10$.
In the former case, the abelian variety has Mumford-Tate group equal
to $\mt^\gen(\ww) \iso \gu_\rat$.  In the latter case, $K$ is
necessarily the compositum of $\rat(\zeta_3)$ and $K_0$, the maximal
totally real subfield of $K$, and the Mumford-Tate group of the
abelian variety is a torus.  The existence of the two kinds of points
over number fields is explored in Corollary \ref{corexistsmallend} and
\cite[Sec.\ 4]{carlsontoledo13}, respectively.

Let $X_0/\cx$ be the elliptic curve with action by $\integ[\zeta_3]$ of
signature $(1,0)$.  

Suppose $B$ is an arbitrary abelian variety.  Then $\End(B\cross B)$
is canonically isomorphic to $\mat_2(\End(B))$.  Via the embeddings 
\begin{diagram}
\integ[\zeta_3] & \rto & \End_\integ(\integ[\zeta_3]) & \rto^\twiddle
& \mat_2(\integ) & \rinject & \mat_2(\End(B)) &
\rto^\twiddle \End(B\cross B),
\end{diagram}
$B\times B$ admits an action by $\integ[\zeta_3]$ of signature $(\dim
B, \dim B)$.

\begin{lemma}
\label{lemclassisog}
If $X$ is a complex abelian fivefold with action by $\integ[\zeta_3]$
of signature $(4,1)$, then $X$ is
isogenous, as an abelian variety with $\integ[\zeta_3]$-action, to an abelian variety of exactly one of the following forms:
\begin{alphabetize}
\item $X_0^{ 3}\cross X_1^{ 2}$, where $X_1$ is an
  elliptic curve which is not isogenous to $X_0$.
\item $X_0^{(5-n)}\cross X_n$, where $1 \le n \le 5$ and $X_n$
  is a simple abelian variety of dimension $n$ with an action by
  $\integ[\zeta_3]$ of signature $(n-1,1)$.
\end{alphabetize}
\end{lemma}

\begin{remark}
The case $n=5$ was already described at the
end of Section \ref{subsecmt}.
\end{remark}

\begin{proof}
Suppose $A$ is an abelian variety with an action by
$\integ[\zeta_3]$.  If $A$ is simple as a $\integ[\zeta_3]$-abelian
variety, then either it is actually simple, or it decomposes as
$A\twiddle B \cross B$.  In the latter case, $B$ does not admit an
$\integ[\zeta_3]$-action, and the signature of the
$\integ[\zeta_3]$-action on $A$ is $(\dim B, \dim B)$.  Moreover, if
$X_1$ and $X_2$ are abelian varieties with $\integ[\zeta_3]$-actions
of respective signatures $(r_1,s_1)$ and $(r_2,s_2)$, then the
signature of $X_1\cross X_2$ is $(r_1+r_2,s_1+s_2)$.  Consequently, if
$X$ is a $\integ[\zeta_3]$-abelian variety of signature $(4,1)$, then
it admits at most one factor which is simple as
$\integ[\zeta_3]$-abelian variety but not as an abelian variety; and
this factor must be the self-product of an elliptic curve.  This is
case (a).  Now suppose that $X$ has no such factor.  A simple complex abelian variety $B$ with
$\integ[\zeta_3]$ action of signature $(r,0)$ necessarily has $r =
\dim B = 1$ \cite[Prop.\ 14]{shimura63}.  Consequently, $X$ has at most
one simple factor of dimension greater than one.  The classification
asserted in the lemma now follows.
\end{proof}

We now take up the task of describing the possibilities for the
endomorphism algebra of a simple abelian variety with $\integ[\zeta_3]$-action.

If $K$ is a CM field, with maximal totally real subfield
$K_0$, let $\u_K$ be the norm-one torus
\[
\u_K = \ker( \res_{K/\rat}\gp_{m,K} \sra{\caln_{K/K_0}}
\res_{K_0/\rat}\gp_{m,K_0}).
\]
Let $(X,\lambda)$ be a simple polarized abelian variety, and let $W =
H^1(X,\rat)$.  The
polarization $\lambda$ induces an involution $(\dagger)$, the Rosati
involution, on $\End(X)_\rat$ and on $\End(W)$.  Suppose that $K$
is contained in the center $C(X)$ of $\End(X)_\rat$, and further
assume that 
$(\dagger)$ restricts to complex conjugation on $K$.  (Given $X$ and
$K$, such a polarization always exists \cite[Lemma 9.2]{kottwitz}.)
Define the group
\begin{align*}
\u_K(X,\lambda) &= \st{\alpha \in \End_K(W): \alpha \alpha^{(\dagger)}
  = 1}
\intertext{with Lie algebra}
\lie u_K(X,\lambda) &= \st{\alpha \in \End_K(W): \forall v,w\in W:
  \lambda(\alpha(v),w)+\lambda(v,\alpha(w))=0}.
\end{align*}
In the absence of a subscript, $\u(X,\lambda)$ will denote
$\u_{C(X)}(X,\lambda)$.  There is
an {\em a priori} inclusion $\liehodge(X) \subseteq \lie u(X,\lambda)$.

Further suppose that $C(X)$ contains $E$.  Then let $\det_{C(X)/E}$
denote the composition
\begin{diagram}
\u(X,\lambda) & \rto^{\det_{C(X)}} &  \u_{C(X)} &
\rto^{\caln_{C(X)/E}} &  \u_E.
\end{diagram}
(Note that we allow the possibility $C(X) = E$, in which case the second
map is an isomorphism; or $2\dim X = 2[C(X):\rat]$, in which case the
first map may be an isomorphism.)

\begin{lemma}
\label{lempossend}
Let $X/\cx$ be a simple abelian variety of dimension $g\le 5$ with an action by
$\integ[\zeta_3]$ of signature $(g -1, 1)$.  If $g=2$, then
$\End(X)_\rat$ is an indefinite quaternion algebra.  If $g\not = 2$, then
$\End(X)_\rat$ is the compositum of $\rat(\zeta_3)$ and a totally
real field of dimension dividing $g$.
\end{lemma}

\begin{proof}
This can be deduced from the usual classification of endomorphism
algebras of abelian varieties (see \cite[Sec.\
2]{moonenzarhinfivefolds}) and the following argument, which is
adapted from  \cite[2.9]{moonenzarhinfourfolds} and shows that
$\End(X)_\rat$ is not a quaternion algebra if $g=4$.  Let $(Y,\lambda)/\cx$ be a simple
polarized abelian variety with $\End(Y)_\rat \iso D$, a quaternion algebra over
$\rat$; then $\dim Y = 2d$ is even.  Now suppose that $\rat(\zeta_3)$
is a subalgebra of $D$ stable under the Rosati involution.  We will
show that the signature of the action 
of $\rat(\zeta_3)$ on $Y$ is $(d,d)$.

As noted above, there is an {\em a priori} inclusion
$\liehodge(Y) \subseteq \lie u_E(Y,\lambda)$.  Moreover, since the center of $D$ is
totally real, $\liehodge(Y)$ is
semisimple \cite[Lemma 1.4]{tankeev79} and we have $\liehodge(Y) \subseteq \lie{su}_E(Y,\lambda)$.  Consequently, $Y$
has signature $(d,d)$ \cite[Lemma 2.8]{moonenzarhinfourfolds}.
\end{proof}

We come to the main result of this section.

\begin{proposition}
\label{prophg}
If $X$ is a complex abelian fivefold with action by $\integ[\zeta_3]$
of signature $(4,1)$, then either:
\begin{alphabetize}
\item $X$ is isogenous to $X_0^{3} \cross X_1^{ 2}$, where
  $X_1$ is an elliptic curve which is not isogenous to $X_0$, 
  and $\hg(X) \iso \u_E \cross \hg(X_1)$; or
\item $X$ is isogenous to $X_0^3 \cross X_2$, where $X_2$ is a simple
  abelian surface.  Then  $D := \End(X)_\rat$ is an indefinite
  quaternion algebra,
\begin{align*}
\hg(X_2) & \iso \st{ x \in (D^{{\rm opp}})\units: x x^{(\dagger)}=1}
\intertext{and}
\hg(X) & \iso \u_E \cross \hg(X_2);
\end{align*}
or
\item $X$ is isogenous to $X_0^{(5-n)}\cross X_n$, where $1 \le n \le 5$
 and $X_n$ is a simple abelian variety of dimension $n$
  with an action by $\integ[\zeta_3]$ of signature $(n-1,1)$. Then
  $\End(X_n)_\rat$ is a CM field which is the compositum of $E$ and a 
  totally real field.  Let
  $\lambda_n$ be a polarization on $X_n$ for which the associated
  Rosati involution induces complex conjugation on $\End(X_n)_\rat$.  Then 
\[
\hg(X_n) \iso \u(X_n,\lambda_n).
\]
\begin{enumerate}
\def\theenumii{\roman{enumii}}
\item If $n = 1$, then
\[
\hg(X) \iso \u_E.
\]
\item If $n = 3$, then
\[
\hg(X) \iso \st{(u,v) \in \u_E \cross \u(X_n,\lambda_n) : u \cdot
  \det_{C(X)/E}(v) = 1}. 
\]
\item If $n = 4$, then
\[
\hg(X) \iso \st{(u,v) \in \u_E \cross \u(X_n,\lambda_n) : u^2 \cdot
  \det_{C(X)/E}(v)=1}.
\]
\item If $n = 5$, then $X = X_n$ and thus
\[
\hg(X) \iso \u(X,\lambda).
\]
\end{enumerate}
\end{alphabetize}
\end{proposition}

\begin{remark}
In part (c), a suitable choice of polarization $\lambda_n$ always
exists \cite[Lemma 
9.2]{kottwitz}.  By definition, the abstract isomorphism class of
$\hg(X_n)$ is independent of the choice of polarization.
\end{remark}

\begin{proof}
The Hodge group of an elliptic curve is either $\SL_2$ or a norm-one
torus $\u_K$ for some quadratic imaginary field $K$.  In particular,
$\hg(X_0) \iso \u_E$.

Suppose $Y_1, \cdots, Y_m$ are abelian varieties with $\hom(Y_i,Y_j) =
\st 0$, and $e_1, \cdots, e_m$ are natural numbers.  Then
$\hg(Y_1^{ e_1} \cross \cdots \cross Y_m^{ e_m}) \iso
\hg(Y_1 \cross \cdots \cross Y_m)$.  Recall the coarse classification
established in Lemma \ref{lemclassisog}.  In case (a) of that lemma, we see that $\hg(X) \iso
\hg(X_0\cross X_1)$.  Moreover, since
$\hom(X_0,X_1) = \st 0$, $\hg(X_0\cross X_1) \iso \hg(X_0)\cross
\hg(X_1)$.

Now suppose that $X$ falls in case (b) of Lemma \ref{lemclassisog}.
The assertions about $\End(X_n)$ follow from Lemma \ref{lempossend},
while information about $\hg(X_n)$ can be read off from \cite[Sec.\
2]{moonenzarhinfivefolds} (for $n = 2$, $3$ and $5$) and \cite[7.4(i),
7.5(i), 7.6(i)]{moonenzarhinfourfolds} (for $n=4$).

If $n \le 4$, then $\hg(X) \iso \hg(X_0 \cross X_n)$.

If $n=1$, $X_1$ is in fact isogenous to $X_0$ (but
equipped with the conjugate action $\bar\iota$ of $\calo_E$).
Consequently, $\hg(X_0^{ 4} \cross X_1) \iso \hg(X_0^{ 5})
\iso \hg(X_0)$.

If $n=2$, then $\hg(X) \iso \hg(X_0) \cross \hg(X_n)$ (\cite[3.2(ii)]{moonenzarhinfivefolds}.

If $n=3$ or $n=4$, then $\hg(X_0 \cross
X_n)$ is calculated in \cite[5.3]{moonenzarhinfivefolds}  and 
\cite[5.11]{moonenzarhinfivefolds}, respectively.

If $n = 5$, then $\hg(X) = \hg(X_n)$ has already been calculated.
\end{proof}

\begin{remark}
We have phrased our answer in terms of Hodge groups, simply because
one often has $\hg(X\cross Y) \iso \hg(X)\cross \hg(Y)$, while this is
never true for Mumford-Tate groups.  Still, it is no more difficult to
calculate $\mt(X)$.
Recall that
there is a $\rat$-linear determinant $\det_\rat: \mt(X) \ra \gp_m$, whose
kernel is precisely $\hg(X)$.  For example, $\mt(X_0)$ is the torus
$\res_{E/\rat} \gp_m$, and in this case $\det_\rat$ coincides with the
$E/\rat$-norm. In case (a) of Proposition \ref{prophg}, if $X_1$ has endomorphism ring $\integ$,
then $\mt(X_1) \iso \gl_2$, and
\[
\mt(X) \iso \st{(\alpha,\beta) \in E\units \cross \gl_2 :
  \det_\rat(\alpha) = \det_\rat(\beta)},
\]
while if $\End(X_2) \iso K$ an imaginary quadratic field, then
\begin{align*}
\mt(X) &\iso \st{(\alpha,\beta) \in E\units \cross K\units:
  \det_\rat(\alpha) = \det(\beta)} \\
&= \st{(\alpha,\beta) \in E\units \cross K\units:
  \caln_{E/\rat}(\alpha) = \caln_{K/\rat}(\beta)}.
\end{align*}
\end{remark}

\subsection{Endomorphism rings}
\label{subsecendring}

Lemma \ref{lemclassisog} implies that an abelian fivefold with
$\integ[\zeta_3]$-action of signature $(4,1)$ has at most one simple isogeny factor of
dimension $n>1$.  Lemma \ref{lempossend} shows that (if $n\not = 2$) then the endomorphism
algebra of such a factor is the compositum of $\rat(\zeta_3)$ and
$K_0$, a totally real field of degree $m$, where $m|n$.  The goal of
the present section is to show (Proposition \ref{propshape}) that every combination of $m$ and $n$
allowed by Lemmas \ref{lemclassisog} and \ref{lempossend} occurs for
the abelian variety attached to some cubic surface.

The crux is to construct principally polarized abelian varieties with
action by $\integ[\zeta_3]$, as opposed to a smaller order.  (The fact
that $\stk N = \tau(\stk S)$ is open in $\stk M$ then implies that, possibly after an
isogeny which still preserves the $\integ[\zeta_3]$-action and
principal polarization, such an abelian variety is attached to a
smooth cubic surface.)  As noted in the introduction, in general it is
delicate to adjust a polarized abelian variety in its isogeny class in
order to simultaneously enlarge its endomorphism ring and induce a 
principal polarization.  See \cite{wilsonrm} for a discussion of this problem.

In fact, it is no harder to work in a somewhat more general context
than that narrowly required by Lemma \ref{lempossend}.  We thus
introduce the following objects and notation, which will be in force
throughout the present subsection.  Let $K_0$ be a totally real field
of degree $[K_0:\rat] = m$.  To make progress on the task at hand, we
will assume:
\begin{equation}
\label{thecond}
\parbox{4in}{$K_0$ is a totally real field such that the inverse different $\cald_{K_0/\rat}\inv$ is a square in the strict class group $\cl^+(K_0)$ of $K_0$.}
\end{equation}
Note that this hypothesis is satisfied if the strict class group of
$K_0$ has odd order, and in particular if it is trivial; a complete
characterization of \eqref{thecond} may be found in \cite[Prop.\ 3.1]{grosssigns}.
Fix an ordering of the distinct
embeddings $\sigma_1, \cdots, \sigma_m: K_0\inject \real$.

Now let $E$ be a quadratic imaginary field; fix an embedding
$\iota:E\inject \cx$, and let $K$ be the compositum $K = K_0E$.  All
polarizations of abelian varieties with action by $\calo_K$ considered
in this section will be compatible with that action, in the sense that
the Rosati involution induces complex conjugation on $\calo_K$.

A CM type (for $K$) is a set $\Sigma$ of embeddings $\phi:K \ra \cx$
such that $\phi \in \Sigma$ if and only if $\bar\phi\not\in\Sigma$,
where $\bar\cdot$ denotes complex conjugation.  There is a bijection
between subsets $S\subset\st{1, \cdots, m}$ and CM types for $K$.
Indeed, let $\tau_i:K\inject \cx$ be the embedding such that
$\tau_i\rest E= \iota$ and $\tau_i\rest{K_0} = \sigma_i$.  Then define
a CM type $\Sigma(S)$ by declaring that $\tau_i\in\Sigma(S)$ if and
only if $i\in S$.

Note that if $X$ is an abelian variety with CM by $K$ of type
$\Sigma(S)$, then it supports an $E$-action of signature $(\#S, m -
\#S)$.

\begin{lemma}
\label{lembuildcm}
 Given $S\subset \st{1, \cdots, m}$, there is a
  principally polarized abelian variety $X$ of dimension $m$ equipped
  with an action by $\calo_K$, compatible with the Rosati involution,
  such that $X$ has type $\Sigma(S)$.
\end{lemma}

\begin{proof}
Record the CM type $\Sigma(S)$ as the ordered tuple $\phi_1, \cdots,
\phi_m:K \inject \cx$, where $\phi_i\rest{K_0} = \sigma_i$.  Then $K$
acts on $\cx^m$ by
\[
a \cdot(t_1, \cdots, t_m) = (\phi_1(a)t_1, \cdots, \phi_m(a)t_m).
\]

Recall the following classical construction of analytic families of
Hilbert-Blumenthal abelian varieties (e.g., \cite[Sec.\
2.2]{gorenhbbook}).  Fix fractional ideals $\ida$ and $\idb$ of $K_0$.
Given a point $\underline z = (z_1, \cdots, z_m)$ in the $m$-fold product of the
upper half-plane $\hp^m$, one defines
a lattice $\Lambda_{\underline z} \subset \cx^m$ by
\begin{align*}
\Lambda_{\underline z} &= \st{ \alpha\cdot \underline z + \beta\cdot \underline 1 : \alpha \in \ida, \beta\in
  \idb} \\
& := \st{(\sigma_1(\alpha)z_1+\sigma_1(\beta), \cdots, \sigma_m(\alpha)z_m+\sigma_m(\beta))
  : \alpha \in \ida, \beta \in \idb}.
\end{align*}
Then $X_{\underline z} = \cx^m/\Lambda_{\underline z}$ is a complex analytic torus with an action
by $\calo_{K_0}$.
For a fixed nonzero $\gamma\in (\cald_{K_0}\ida\idb)\inv \subset K_0$, define a pairing
\begin{diagram}
(\ida \oplus \idb) \cross (\ida \oplus \idb) & \rto^{E_\gamma} & \rat \\
(\alpha_1,\beta_1) \cross(\alpha_2,\beta_2) & \rmto &
\tr_{K_0/\rat}(\gamma(\alpha_1\beta_2-\alpha_2\beta_1)).
\end{diagram}
Then $E_\gamma$ takes on values in $\integ$, and is the imaginary part
of the
Riemann form
\[
H_{\gamma,\underline z}((x_1, \cdots, x_m),(y_1, \cdots, y_m)) = \sum_{1 \le i \le m}
\frac{x_i\bar y_i \sigma_i(\gamma)}{\Im(z_i)},
\]
which is (the first Chern class of) an $\calo_{K_0}$-linear
polarization $\lambda_\gamma$ of $X_{\underline z}$.  If one further
assumes that $\ida \idb = \cald_{K_0}\inv$ in $\cl(K_0)^+$, then
$\gamma$ may be chosen so that $\lambda_\gamma$ is principal.

We now specialize this construction to produce the desired $X$.  Let
$\ida = \idb$ be a fractional ideal in $K_0$ such that $\ida \cdot
\ida$ is in the strict ideal class of $\cald_{K_0}\inv$, and choose a
$\gamma$ so that $H_{\gamma, \underline z}$ is a principal polarization
on $\cx^m/\Lambda_{\underline z}$ for each $\underline z\in \hp^m$.

Fix a lattice $\integ \cdot 1 \oplus \integ \cdot z_0\subset \cx$ which is stable
under $\iota(\calo_E)$.  Let $\underline z^* =
(z_0, \cdots, z_0)$, and let $X = \cx^m/\Lambda_{\underline z^*}$.
Then $H_{\gamma,\underline z^*}$ defines a principal polarization, and
in particular $X$ is an abelian variety with an action by
$\calo_{K_0}$.

In fact, $X$ also supports an action by $\calo_E$.  To see this,
suppose $\epsilon \in \calo_E$.  There are integers $a$, $b$, $c$ and
$d$ such that 
\begin{align*}
\epsilon \cdot 1 &= az_0+b \\
\epsilon \cdot z_0 &= cz_0+d.
\end{align*}
Given a lattice element $\alpha \cdot \underline z^* + \beta \cdot
\underline z^* \in \Lambda_{\underline z^*}$, we then have
\begin{align*}
\epsilon \cdot (\alpha \cdot \underline z^* + \beta \cdot \underline
z^*) &=
\alpha \cdot (c \underline z^* + d\underline 1) + \beta \cdot
(a\underline z^*+b\underline 1) \\
&= (c\alpha + a\beta)\underline z^* + (d\alpha + b\beta)\cdot \underline
1 \\
&\in \Lambda_{z^*}
\end{align*}
since $\alpha,\beta \in \ida$.

Moreover, the CM type of the action of $K$ is $\Sigma(S)$.  (This is essentially worked out in \cite[Sec.\
22]{mumfordabvar}, albeit with a different, though commensurable,
choice of lattice since the degree of the polarization is not
germane there.)
\end{proof}

Now consider an abelian variety $Y$ of arbitrary dimension, equipped with an action
by $\calo_K$ such that $1\in \calo_K$ acts as $\id_Y$.  The signature
of this action -- i.e., the isomorphism class of $\Lie(Y)$ as
$\calo_K\tensor\cx$-module -- is described by nonnegative integers
$r_i$ and $s_i$ for $1 \le i \le m$.  Here, $r_i$ is the dimension of
the $\tau_i$-eigenspace of $\Lie(Y)$, and $s_i$ is the dimension of
the $\bar\tau_i$-eigenspace.  As an abelian
variety with action by $\calo_E$, $Y$ has signature $(\sum_{1 \le i
  \le m} r_i, \sum_{1 \le i \le m} s_i)$.

\begin{lemma}
\label{lembuildoksimple}
Let $S^{(1)}, \cdots, S^{(n)}$ be subsets of $\st{1, \cdots, m}$.
\begin{alphabetize}
\item There is a principally polarized abelian variety $Y$ of dimension
$mn$, equipped with an action by $\calo_K$ of signature $(r_1, \cdots,
r_m; s_1, \cdots, s_m)$, where
\begin{equation}
\label{eqcmsig}
\begin{array}{rl}
 r_i &= \#\st{j : i \in S^{(j)}} \\
 s_i &= \#\st{j :i \not \in S^{(j)}}.
\end{array}
\end{equation}

\item Suppose there are indices $1 \le i, j \le m$ such that $r_is_i\not = 0$
  and $r_js_j\not = 1$.  Then there exists a $Y$ as in (a) such that 
$\End(Y) \iso \calo_K$.  In particular, such a $Y$ is simple.
\end{alphabetize}
\end{lemma}

\begin{proof}
For each $1 \le j \le m$, use Lemma \ref{lembuildcm} to construct a
principally polarized abelian variety $X^{(j)}$ with action by
$\calo_K$ of signature $\Sigma(S^{(j)})$.  Then $Y = \cross_{1 \le j
  \le n} X^{(j)}$, equipped with its product principal polarization,
has an $\calo_K$-action whose signature is given by
\eqref{eqcmsig}.   This proves (a).

Now use $H_1(Y,\integ)$ and its Riemann form as input to Shimura's
construction in \cite[Sec.\ 2]{shimura63}.  (The relevant case is also
cleanly documented in \cite[Sec.\ 9.6]{birkenhakelange}.)  The output
of this construction is a family of principally polarized abelian
varieties on which $\calo_K$ acts.  The endomorphism ring of a
sufficiently general member of this family is $\calo_K$, and not
larger, if the signature satisfies the combinatorial
constraint given in (b) (\cite[Thm.\ 5]{shimura63}; see also
\cite[Thm.\ 9.9.1]{birkenhakelange}).
\end{proof}

\begin{lemma}
\label{lemcmsig}
Let $K_0$ be a totally real field of degree $m$ which satisfies
\eqref{thecond}.  Let $E$ be a quadratic imaginary field and let $K = K_0E$.
Given positive integers $n$, $r$ and $s$ satisfying $r+s = mn$,
suppose that either:
\begin{alphabetize}
\item $n\ge 3$; or
\item $n = 2$ but either $m\ge 3$ or $m=2$ and $\st{r,s} = \st{3,1}$; or
\item $n=1$ and $\gcd(r,s)=1$.
\end{alphabetize}
Then there exists a principally polarized abelian variety $Y$ of
dimension $mn$ such that $\End(Y) = \calo_K$ and the signature of $Y$,
as $\calo_E$-abelian variety, is $(r,s)$.
\end{lemma}

\begin{proof}
Given Lemma \ref{lembuildoksimple}, it suffices to show that there
exist subsets $S^{(1)}, \cdots, S^{(m)}$ such that, if $r_i$ and $s_i$
are calculated as in \eqref{eqcmsig}, then there are
indices $i$ and $j$ such that $r_is_i\not = 0$ and $r_js_j\not = 1$.
We briefly indicate how this may be accomplished.  Note that if
$S^{(j)}$ is represented as a row of length $m$ with entries $+$ and
$-$, then specifying types $S^{(1)}, \cdots, S^{(n)}$ is equivalent to 
labeling each cell of an $n\times m$ matrix with either $+$ or $-$.
Then $r_i$ is the number of $+$ symbols in column $i$, and $s_i$ is
defined analogously.

We may and do assume that $r \ge s$.
Suppose $n\ge 3$.   Any labeling of the $n\times m$ matrix in which
both symbols appear in the first column will suffice, since then $r_1,
s_1 \ge 1$ and $r_1+s_1 = n \ge 3$.  This proves (a).

Now suppose $n=2$.  Under the additional hypotheses, $r \ge 3$ and $s
\ge 1$.  Consequently, we may choose a labeling in which the first two
columns are given by
\[
\begin{matrix}
++\\+-
\end{matrix}.
\]
This proves (b).

For part (c), the hypothesis implies that any CM type $S$ with
$\calo_E$-signature $(r,s)$ is simple, and thus so is an associated
abelian variety (\cite[p. 213]{mumfordabvar}).
\end{proof}

If one is willing to believe the folklore conjecture that a positive
proportion of totally real fields have strict class number one,
there are many examples of the behavior allowed by Lemma
\ref{lempossend}:

\begin{proposition}
\label{propshape}
Let $K_0$ be a totally real field of degree $m$ which satisfies
\eqref{thecond}, and let $K = K_0(\zeta_3)$. Let $n$ be a natural number such that $mn\not = 2$.
\begin{alphabetize}
\item There exists a simple principally polarized abelian variety $X$
  of dimension $mn$ such that $\End(X) \iso \calo_K$ and, as
  $\integ[\zeta_3]$-abelian variety, $X$ has signature $(mn-1,1)$.

\item If furthermore $1 \le mn \le 5$, then there exists a smooth cubic
  surface $Y$ for which the associated abelian variety $Q(Y)$ is
  isogenous to $X_0^{5-mn} \cross X$, where $X_0$ is the elliptic
  curve with $\End(X_0) \iso \integ[\zeta_3]$.
\end{alphabetize}
\end{proposition}
\begin{proof}
Part (a) is a special case of Lemma \ref{lemcmsig}.  For part (b),
consider the moduli point of $X_0^{5-mn}\cross X$ in $\stk M(\cx)$.
Its Hecke orbit is analytically dense in $\stk M_\cx$, and in
particular include points in the (Zariski open) subset $\tau(\stk
S_\cx) = \stk N_\cx$.
\end{proof}

\begin{remark}
Carlson and Toledo have also addressed the question of when an abelian
variety with complex multiplication admits a principal polarization.
In particular, \cite[Prop.\
5]{carlsontoledo13} proves a result similar to Proposition
\ref{propshape} in the special case $n=1$, albeit with an additional
positivity condition on $\ida$ and less explicit control over $r$ and $s$.
\end{remark}

For the sake of completeness (and at the suggestion of the referee),
we observe that the case $g=2$ of Lemma \ref{lempossend} occurs, too.

\begin{lemma}
Let $D$ be an indefinite, nonsplit quaternion algebra over $\rat$
which contains $E$.  Then there exists a principally polarized abelian
surface $X$ with action by $\calo_E$ such that $\End(X)_\rat
\iso D$ and $X$, as $\calo_E$-abelian variety, has signature $(1,1)$.
\end{lemma}

\begin{proof}
Let $D$ be an indefinite quaternion algebra over $\rat$ which contains
$E$; choose an involution $(*)$ on $D$ which induces complex
conjugation on $E$. 
Let $\calo_D$ be any
maximal order of $D$ which contains $\calo_E$.  There is a
(one-dimensional, and in particular nonempty) Shimura
curve whose points parametrize abelian
surfaces $X$ with an action by $\calo_D$. On each $X$
there exists a unique principal polarization for which the induced
Rosati involution coincides with $(*)$ \cite[Prop.\
III.1.5]{boutotcarayol}.  If $X$ is sufficiently general, then $\End(X)_\rat$ is $D$, and not larger (again,
see \cite[Thm.\ 5]{shimura63} or \cite[Thm.\ 9.9.1]{birkenhakelange}).  
It remains to
verify the signature condition.   By the Skolem-Noether theorem, there exists some
$a \in D$ such that, for $b \in E \subset D$, $b^{(*)} = aba\inv$.
The action of $b$ on $\Lie(X)$ thus exchanges the two eigenspaces for
the action of $\calo_E$, and each is necessarily one-dimensional.
\end{proof}

\section{Arithmetic applications}
\label{secmono}

We start by calculating the geometric monodromy group of the first
\'etale cohomology of the tautological abelian scheme over $\tilstk T$
and its pullback to $\tilstk S$.  This is an arithmetic analogue of
the topological monodromy calculation by Allcock, Carlson and Toledo
(Proposition \ref{propactmono}).    Subsequently, this is applied to
two problems in the arithmetic of cubic surfaces.

First, we show that there are cubic surfaces $Y$ and threefolds $Z$
with {\em rational} coefficients for which $\End(Q(Y))$ and
$\End(P(Z))$ are as small as possible.

Second, the arithmetic of the 27 lines on a cubic surface is a natural
locus of inquiry \cite{jordan70}.  We close the paper by calculating
the field of definition of the lines on a cubic surface over a field
of positive characteristic.

\subsection{$\integ_\ell$-monodromy}
\label{subsecmono}

There is a natural abelian scheme over $\tilstk
T$ which pulls back to one on $\tilstk S$.  The first \'etale
cohomology group of this abelian scheme gives a lisse
$\integ_\ell$-sheaf on the base, which makes sense in any
characteristic other than $\ell$.  The main goal of the present section is to compute
the geometric monodromy groups of these $\integ_\ell$ sheaves.  (In
characteristic zero, this recovers the fact (Proposition
\ref{propactmono}) that the connected component of the topological
monodromy group of $\tilstk S$ is $\su$.)

Let $\calf/S$ be a lisse $\integ_\ell$-sheaf of rank $r$ over a
connected scheme (or stack)
$S$ on which $\ell$ is invertible.  After choosing a geometric point
$s\in S$, such an object is equivalent to a
representation $\rho_{\calf/S}: \pi_1(S,s) \ra \aut(\calf_s) \iso
\gl_r(\integ_\ell)$.  Let $\mono(\calf/S)$ be the monodromy group of
$\calf$, i.e., the image of this representation; it is a closed
subgroup of the $\ell$-adic group $\gl_r(\calf_s)$.

Recall that, if $\ell$ is an odd prime, if $H$ is either a special unitary group $\su_{r,s}$ or
symplectic group $\sp_{2g}$ and if $\Gamma \subset H(\integ_\ell)$ is
a closed subgroup, then $\Gamma = H$ if and only if the composition
$\Gamma \inject H(\integ_\ell) \ra H(\integ/\ell)$ is surjective.  In
this way, calculating the monodromy group of a lisse
$\integ_\ell$-sheaf $\calf$ is often tantamount to computing computing the
monodromy group of the $\integ/\ell$-sheaf $\calf\tensor\integ/\ell$.

We start with some geometry.  Let $\stk H_5$ be the moduli space of
hyperelliptic curves of genus five, identified with its image in $\stk
A_5$ under the (classical) Torelli map.

\begin{proposition}
\label{prophyperboundary}
Let $k$ be a field in which $2$ is invertible.
Then the closure of $\til\varpi(\tilstk T_k)$ in $\stk A_{5,k}$ contains $\stk
H_{5,k}$. 
\end{proposition}

\begin{proof}
Since taking Zariski closure commutes with flat base change, we assume
that $k$ is algebraically closed.
Let $X/k$ be an abelian fivefold which is a hyperelliptic Jacobian.
More explicitly, suppose $X$ is the Jacobian of the
hyperelliptic curve branched at points $q_1, \cdots, q_{12} \in
\proj^1$.  Identify this copy of $\proj^1$ with a rational normal
curve $N\subset \proj^4$, and   let $Z_0$ be the secant variety of
$N$.  Then $Z_0$ is a cubic threefold, singular exactly along $N$.
Let $G_0$ be a defining cubic form for this singular threefold.  Let
$Z_1$ be any {\em smooth} cubic threefold passing through each of the
$p_i$, transverse to $N$, and let $G_1$ be a defining cubic form for $Z_1$.

Consider the cubic threefold $Z$ over $k\powser t$ with defining
equation $t^2G_1+G_0$.  To prove the proposition, it suffices to show that
the closure of $P(Z_{k\laurser t})$ in $\stk A_5$ contains the moduli
point of $X$.

If $k$ has characteristic zero, this is the main result of
\cite{collino82}; the claim is proved by transcendental means.

If $k$ has positive characteristic, let $R$ be a mixed characteristic
discrete valuation ring with residue field $k$, and let $K= \Frac R$
be its fraction field.  Lift $N$ to a rational normal curve $\hat N
\subset \proj^4_R$, and for $i \in \st{1, \cdots, 12}$ let $\hat q_i
\in \hat N(R)$ be a section lifting $q_i$.  Let $\hat Z_0/R$ be the
secant variety of $\hat N$.  Similarly, let $\hat
Z_1/R$ lift $Z_1$ and pass through each $\hat q_i$ transversally to
$\hat N$.  For $i = 1,2$,  let $\hat G_i$ be the cubic form with
coefficients in $R$ defining $\hat Z_i$, and consider the cubic form
$t^2\hat G_1 + \hat G_0$ over $R\powser t$.

Let $\xi: \spec R\powser t\inject \barstk A_5$ be the closure of the
moduli point of $P(Z_{K\laurser t})$ in the minimal compactification
of $\stk A_5$.  The result in characteristic zero shows that $\xi_K:
\spec K \ra K\powser t \inject \barstk A_5$ is the moduli point of the
Jacobian $\hat X_K$ of the hyperelliptic curve branched over (the generic fibers
of) $\hat q_1, \cdots, \hat q_{12}$.  Since the closure $\xi_R$ of
$\xi_K$ contains the moduli point of $X$, $\xi$ is in fact a morphism
$\xi:\spec R \powser t \inject \stk A_5$.  Moreover, $\xi_k$ is in the closure of $\xi_{k\laurser t}$, which
means that $P(Z_{k\laurser t})$ is the generic fiber of a deformation
of $X$.
\end{proof}

The next result says that certain natural monodromy groups attached to
cubic surfaces and threefolds are as large as possible, given the
obvious constraints imposed by polarization and, for cubic surfaces,
an action by $\integ[\zeta_3]$.

\begin{remark}
\label{remmonoqlbar}
The $\bar\rat_\ell$-monodromy group of the universal family of
hypersurfaces of given degree and dimension is known, thanks to work
of Deligne \cite[4.4.1]{deligneweil2}.
The improvement in Theorem \ref{thmono} is that one knows the
$\integ_\ell$-monodromy exactly, without passage to the Zariski
closure or extension to $\bar\rat_\ell$-coefficients.
\end{remark}

Recall that the tautological abelian fivefold and cubic threefold are
given by $f:\calx \ra \stk A_5$ and $h:\calz \ra \tilstk T$.

\begin{theorem}
\label{thmono}
Let $k$ be an algebraically closed field in which $2$ is invertible,
and let $\ell$ be an odd rational prime invertible in $k$.  Then:
\begin{alphabetize}
\item $\mono(R^1 f_*\integ_\ell \ra \til\varpi(\tilstk T_k)) \iso
  \sp_{10}(\integ_\ell)$;
\item $\mono(R^3 h_*\integ_\ell(1) \ra \tilstk T) \iso
  \sp_{10}(\integ_\ell)$;
\item $\mono(\til \phi^* R^3 h_*\integ_\ell(1) \ra \tilstk S) \iso
  \u^*(\integ_\ell)$ if $\ell\not = 3$ and ${\rm char}(k)\not = 3$.
\end{alphabetize}
\end{theorem}

\begin{proof}
For (a), if $S\subset \stk A_{5,k}$ is any subspace, there is an {\em
  a priori} constraint $\mono(R^1f_*\integ_\ell \ra S) \subset
\sp_{10}(\integ_\ell)$.  By semicontinuity and Proposition
\ref{prophyperboundary}, $\mono(R^1f_*\integ_\ell \ra \til\varpi(\tilstk
T_k))$ contains $\mono(R^1f_*\integ_\ell \ra \stk H_5)$.  This last
group is known to be $\sp_{10}(\integ_\ell)$ (e.g., \cite[Thm.\
3.4]{achterpriesmono}, \cite{hall08}).

Part (b) then follows from the isomorphism
$\til\varpi^*R^1f_*\integ_\ell \ra R^3h_*\integ_\ell(1)$ established
in
\cite[Prop.\ 3.6]{achtercubmod}
and the fact that $\til\varpi$ is a
fibration. 

For part (c), the assertion that $\mono(R^1f_*\integ_\ell \ra \stk M)
\iso \u^*(\integ_\ell)$ follows from the transcendental uniformization
of $\stk M_\cx$ and the theory of arithmetic compactifications; the
details are provided in Section \ref{subsecunitarymono}, and in
particular Lemma \ref{lemunitarymono}.

Since $\tau:\stk S_k \ra \stk M_k$ is an open embedding, each
irreducible component of the PEL Shimura variety $\stk M_k^\ell$
analyzed in Lemma 
\ref{lemunitarymono} pulls back to an irreducible cover of $\stk S_k$,
and in particular $\mono(\til\phi^*R^3h_*\integ_\ell(1) \ra \stk S)
\iso \u^*(\integ_\ell)$.  Now use the fact that $\til S \ra S$ is a
fibration with connected fibers.
\end{proof}

\subsection{Monodromy groups for unitary Shimura varieties}
\label{subsecunitarymono}

The goal of this section is to calculate the $\ell$-adic monodromy for
a unitary Shimura variety attached to $\integ[\zeta_3]$.   It is no
harder to allow an essentially arbitrary signature condition on the
$\integ[\zeta_3]$-action, and this is useful for applications outside the
present paper \cite{achterpriesmono}. 

Some notation and background on unitary groups will be useful.   For
positive integers $r$ and $s$, let $g = r+s$ and let $V_{r,s}$ be the module
$\calo_E^{\oplus g}$, with standard basis $e_1, \cdots, e_g$.  Endow
$V_{r,s}$ with the Hermitian pairing
\[
\hang{\sum x_je_j,\sum y_je_j} = \sum_{1 \le j \le r}x_j\bar y_j -
\sum_{r+1 \le j \le g}x_j \bar y_j.
\]
As in Section \ref{subsecunitary}, define $\gu_{r,s} =
\gu(V_{r,s},\hang{\cdot,\cdot})$, and define $\su_{r,s}$ and $\u_{r,s}$ analogously.
If $\T$ and $\T_1$ are the tori defined in
\eqref{eqdeft}-\eqref{eqdeftone}, then these group schemes again sit
in exact sequences \eqref{diagesunitary}.
Finally, as in \eqref{eqdefustar}, let $\u_{r,s}^*$ be the extension of the finite group scheme
$\mmu_6$ by $\su_{r,s}$, so that  $\u_{r,s}^*(\integ/N)$ is the
extension of the global units $\calo_E\units$ by $\su_{r,s}(\integ/N)$.

If $\ell$ is inert in $\calo_E$, then
$\su_{r,s}(\integ/\ell)$ is isomorphic to the classical group
$\su(g,\ff_\ell)$, while if $\ell$ splits, then $\su_{r,s}(\integ/\ell)
\iso \SL(g,\ff_\ell)$.

\begin{lemma}
\label{lemimageunitary}
Let $\ell$ be a prime not equal to $3$.  The natural map
$\su_{r,s}(\integ) \ra \su_{r,s}(\integ/\ell)$ is surjective.
\end{lemma}

\begin{proof}
This is presumably well-known. For want of a reference, we provide a
proof here, proceeding by induction on $r+s$. 

Suppose $r = s = 1$.  Note that, for any ring $R$, an element $v_1e_1+v_2e_2 \in
V\tensor R$ is isotropic if and only if $v_2$ is either $v_1$ or $\bar v_1$.
Now, $\su_{1,1}(\integ/\ell)$ is
generated by isotropic transvections, i.e., maps of the form
\begin{diagram}
\tau_{u,a}:v & \rmto & a\hang{v,u}u
\end{diagram}
where $u\in V\tensor\integ/\ell$ is isotropic and $a
\in \calo_E/\ell$ satisfies $a = - \bar a$.   (This is well-known
if $\calo_E/\ell \iso \ff_{\ell^2}$ is a field; if $\ell$ splits in
$\calo_E$, this is is still true, since the quotient ring is semilocal
\cite[p.\ 535]{hahnomeara}.)
It is clear that such a
$u$ admits an isotropic lift $\til u$ to $V$ itself.  Similarly, let
$\til a\in \calo_E$ be a totally imaginary element (i.e., $\til a +
\bar{\til a} = 0$) whose reduction modulo $\ell$ is $a$.  Then 
\begin{diagram}
\til \tau_{\til u, \til a}: v & \rmto &  \til a \hang{v, \til u}\til u
\end{diagram}
is an element of $\su(V_{1,1})$ whose reduction modulo $\ell$ is
$\tau_{u,a}$.  Thus, the image of $\su_{1,1}(\integ) \ra
\su_{1,1}(\integ/\ell)$ contains a generating set for the range.

Now suppose the claim is known in dimension $g-1$.  We may and do suppose without
loss of generality that $r \ge 2$.  Let $W_1 \subset V$ be the
$\calo_E$-span of $e_1, e_3, e_4,\cdots, e_g$, and let $W_2$ be the
$\calo_E$-span of $e_2, e_3, \cdots, e_g$.  Then $\su(W_i) \iso
\su_{r-1,s}$; by induction, $\su(W_i)(\integ)$ surjects onto
$\su(W_i)(\integ/\ell)$.  It is not hard to check (e.g., \cite[Lemma
3.2]{achterpriesmono}) that, independent of the splitting behavior of
$\ell$ in $\calo_E$,  the only subgroup of
$\su_{r,s}(\integ/\ell)$ which contains each $\su(W_i)(\integ/\ell)$
is $\su_{r,s}(\integ/\ell)$ itself.
\end{proof}

To ease notation, we now {\em fix} positive integers $r$ and $s$, and
write $\su$ for $\su_{r,s}$, etc.  Let $\stk M = \stk
M_{\integ[\zeta_3,1/3](r,s)}$ be the moduli stack over 
$\integ[\zeta_3,1/3]$ of triples $(X,\iota,\lambda)$ where $X$ is an
abelian scheme of dimension $g=r+s$;
$\iota:\integ[\zeta_3]\inject \End(X)$ is an action with signature
$(r,s)$; and $\lambda$ is a principal polarization compatible with
$\iota$.  Of course, the datum $(r,s) = (4,1)$ is the one relevant,
via $\tau$, to the study of cubic surfaces.

\begin{lemma}
\label{lemunitarymono}
Let $k$ be an algebraically closed field in which $3$ is invertible,
and let $\ell \ge 5$ be a rational prime invertible in $k$.  Let $t\in
\stk M(k)$ be a geometric point.  Then the image of the monodromy
representation
\begin{diagram}
\pi_1(\stk M_k, t) &\rto & \aut(\calx_t[\ell])
\end{diagram}
is $\u^*(\integ/\ell)$.
\end{lemma}

\begin{remark}
\label{remwhoops}
Suppose that $\frac{g-1}{3} \le r, s \le \frac{2g+1}{3}$.  Then some
$\integ[\zeta_3]$-abelian varieties of signature $(r,s)$ are actually
Jacobians of cyclic triple covers of the projective line \cite[Lemma
2.9]{achterpriesmono}.  Lemma \ref{lemunitarymono} is compatible with
\cite[Thm.\ 5.4]{katzact} but not, alas, \cite[Thm.\
3.8]{achterpriesmono}; in spite of its claims to the contrary, \cite{achterpriesmono} only calculates the
connected component of the monodromy group.
\end{remark}

\begin{proof}
We wish to describe the action of $\pi_1(\stk M_k,t)$ on
$\calx_t[\ell]$.  Equivalently, we wish to describe the action of
$\pi_1(\stk M_k, t)$ on the set of {\em principal level $\ell$
  structures} on $\calx_t[\ell]$.  A principal level $\ell$ structure on
a polarized abelian scheme with $\calo_E$-action
$(X/S,\iota,\lambda) \in \stk M(S)$ is the data of an
$\calo_E$-equivariant isomorphism $\alpha:(V\tensor \calo_E/\ell)_S \sriso
 X[\ell]$ and an isomorphism $\nu:\integ/\ell(1)_S \sriso
 \mmu_{\ell,S}$ such that the following diagram commutes:
\begin{diagram}
(V\tensor\calo_E/\ell)_S \cross_S (V\tensor\calo_E/\ell)_S &
\rto^{\ang{\cdot,\cdot}} & \integ/\ell(1)_S \\
\dto<{\alpha\cross\alpha} &&\dto>{\nu} \\
X[\ell] \cross_S X[\ell] & \rto^{e^\lambda} & \mmu_{\ell,S}.
\end{diagram}
(Actually, we have chosen and then suppressed an isomorphism $\integ \sriso
\integ(1)$, so that $\ang{\cdot,\cdot}$ takes values in $\integ(1)$.)
This formulation of level structure is taken from
\cite[1.3.6.1]{lanbook}; it is equivalent \cite[1.4.3.4]{lanbook} to the perhaps more frequent
``rational level structure on abelian varieties up to isogeny'' used
in, e.g., \cite{kottwitz}, and more convenient for the application
here.   

Let $\stk M^\ell$ be the moduli space of principally polarized abelian
varieties with $\integ[\zeta_3]$-structure of signature $(r,s)$ and
principal level $\ell$ structure. 
Then, away from characteristic $\ell$, $\stk M^\ell$ is a 
$\gu(\integ/\ell)$-torsor over $\stk M$.  (In fact, the choice of an
isomorphism $\alpha$ determines the isomorphism $\nu$; but explicitly tracking both pieces of information makes it clear that $\stk M^\ell$ really is a $\gu(\integ/\ell)$-torsor.)  

Let $t = (X,\iota,\lambda) \in \stk M(k)$, and let $\til t =
(X,\iota,\lambda,(\alpha,\nu))\in \stk M^\ell(k)$ be a point lying
over it. 
The action
of $\pi_1(\stk M_k, t)$ on $\calx_t[\ell]$ is the action of
$\pi_1(\stk M_k, t)$ on $\til t$.  In particular, the
mod-$\ell$ monodromy group is (noncanonically) isomorphic to the
group of automorphisms of an irreducible component of $\stk M^\ell_k$
over $\stk M_k$.  Thanks to the existence of arithmetic toroidal
compactifications \cite[6.4.1.1]{lanbook} and the usual Zariski
connectedness argument (e.g., \cite[6.4.12]{lanbook}), it suffices to carry out this last calculation
in the special case $k = \cx$.

There is a double quotient of the ad\`elic space $\gu(\aff)$ which
encodes the structure of the set of components of $\stk M^\ell_\cx$.  This is
in \cite[2.9]{delignecorvallis}, but we follow here the
more expansive treatment in \cite[Sec.\ 5]{milneintroshim}.  

Let $\aff_\fin$ denote the ring of finite ad\`eles, and let $\K_\ell = \st{
  \alpha \in G(\hat\integ): \alpha \equiv \id \bmod \ell}\subset G(\aff_\fin)$.  
Then $\stk M^\ell_\cx$ admits the
uniformization
\[
\stk M^\ell_\cx = \gu(\rat)\bs {\mathbb X} \cross \gu(\aff_\fin)/\K_\ell,
\]
where the Hermitian symmetric domain ${\mathbb X}$ is a certain $\gu(\real)$-conjugacy class of homomorphisms
$\res_{\cx/\real}\gp_{m,\cx} \ra \gu_\real$.

Now, $\T$ is the maximal abelian quotient of $\gu$, $\T(\real)\iso
\cx\units$ is (topologically) connected, and $\su$ is
simply connected.  Consequently, there is a canonical bijection between
irreducible components of $\stk M^\ell_\cx$ and the finite group
\begin{equation}
\label{eqdq}
\T(\rat)\bs \T(\aff_\fin)/\nu(\K_\ell)  = E\units \bs \aff_{E,\fin}\units /
\nu(\K_\ell)
\end{equation}
\cite[5.17 and p.311]{milneintroshim}.

We compute this last quotient.  The compact group $\nu(\K_\ell)$ is
simply the group of finite id\`eles in $E$ which are everywhere integral and congruent to one modulo
$\ell$. Indeed, let $\idl$ be the modulus associated to $\ell$; then
$\nu(\K_\ell)$ is the finite part of the congruence group
$\aff_{E,\idl}\units$.  We thus recognize the quotient in
\eqref{eqdq} as 
\[
\Cl_E^\idl = \frac{\aff_{E}\units/E\units}{E\units
  \aff_{E,\idl}\units/E\units},
\]
the ray class group modulo $\idl$.  The ray class group may, of
course, be computed via ideals, instead of id\`eles.  Since $E$ is
totally imaginary, $\Cl_E^\idl$ fits into an exact sequence
\cite[Thm.\ 5.1.7]{milnecft} 
\begin{diagram}
0 & \rto & \calo_E\units/\calo_{E,\idl}\units & \rto &
(\calo_K/\ell)\units & \rto & \Cl_E^\idl & \rto & \Cl_E & \rto & 0,
\end{diagram}
where $\Cl_E$ is the class group of $E$ and $\calo_{E,\idl}\units$ is
the group of global units which are one modulo $\ell$.  These two
groups are trivial because, respectively, $\integ[\zeta_3]$ has unique
factorization and $\ell \ge 5$.  Consequently,  $\Cl_E^\idl \iso
(\calo_E/\ell)\units/\calo_E\units$.

In particular, the irreducible components of $\stk M^\ell_\cx$ are
parametrized by the cosets $\gu(\integ/\ell)/\u^*(\integ/\ell)$.
Therefore, the stabilizer of each component of $\stk M^\ell_\cx$ is
$\u^*(\integ/\ell)$, as claimed.
\end{proof}

\subsection{Small endomorphism rings over $\rat$}
\label{subsecenddescent}

Carlson and Toledo have given explicit examples of cubic surfaces
$Y/K$ over number fields such that $\End(Q(Y))$ is large.  The
following result shows that the generic (i.e., small) endomorphism
ring is also attainable over number fields.  Note that Corollary
\ref{corexistsmallend} is not an immediate consequence of Lemma 
\ref{lemgenmt}.  If $T$ is a variety over $\rat$ and $\ww \ra T_\cx$
is a variation of Hodge structure, cardinality considerations alone do
not exclude the possibility that the Hodge-generic locus of $\ww \ra T_\cx$ avoids
all of $T(\bar\rat)$.

\begin{corollary}
\label{corexistsmallend}
\begin{alphabetize}
\item There exists a cubic threefold $Z/\rat$ such that
  $\End_{\bar\rat}(P(Z)) \iso \integ$.
\item There exists a cubic surface $Y/\rat$ such that
  $\End_{\bar\rat}(Q(Y)) \iso \integ[\zeta_3]$.
\end{alphabetize}
\end{corollary}

\begin{proof}
This essentially follows from Theorem \ref{thmono} and a result of
Terasoma \cite[Thm.\ 2]{terasoma85}.  Full details are provided here
for part (b); part (a) is entirely analogous.

Fix a rational prime $\ell \ge 5$  which is inert in $\rat(\zeta_3)$.
As in Section \ref{subsecunitary}, let $\stk M^\ell$ be the moduli
space of principally polarized abelian
varieties with $\integ[\zeta_3]$-structure of signature $(4,1)$ and
principal level $\ell$ structure.  Let $K$ be any number field such
that every irreducible component of $\stk M_K^\ell$ is geometrically
irreducible, and let $\stk M_K^{\ell,0}$ be one such component.

 Let $\tilstk
M_K^{\ell,0}$ be the pullback of $\stk M_K^{\ell,0}$ to $\tilstk
S_K$.  Since $\tau$ is an open immersion and $\tilstk S \ra \stk S$ is
a fibration, $\tilstk M_K^{\ell,0} \ra \tilstk S$ is an absolutely
irreducible Galois cover of varieties, with Galois group
$\u^*(\integ/\ell)$ (Lemma \ref{lemunitarymono}).
Since $\tilstk S_K$
is an open subset of a projective space, by Hilbert's irreducibility
theorem \cite[Prop.\ 3.3.1]{serre_galtheory} there is a thin set
$A\subset \tilstk 
S(K)$ such that if $s\in \tilstk S(K) \setcomp A$, then $\tilstk M_s^{\ell,0} \ra s
\iso \spec K$ is again irreducible.

For such an $s$, the absolute endomorphism ring of $Q(\stk Y_s)$
is $\integ[\zeta_3]$.  Indeed, the image of the Galois representation
$\pi_1(s, \bar s) \iso \gal(K) \ra \aut(T_\ell Q(\stk Y_s))$
contains $\su(\integ_\ell)$, and thus for any finite extension
$L/K$, the image of $\gal(L)$ contains an open subgroup of
$\su(\integ_\ell)$.  Since the $\ell$-adic completion of the
endomorphism ring of $Q(\stk Y_s)_L$ is the commutant of the
image of $\gal(L)$, it follows that $\End_L(Q(\stk Y_s)) \iso
\integ[\zeta_3]$.

Moreover, $A\cap \tilstk S(\rat)$ is also thin in $\tilstk S(\rat)$
\cite[Prop.\ 3.2.1]{serre_galtheory}, so there are many suitable cubic
surfaces defined over $\rat$.
\end{proof}

\subsection{The Galois group of the 27 lines}
\label{subsec}

Camille Jordan investigated (\cite[III.II.V]{jordan70}; see also
\cite{harrisenum}) the structure of the field of definition 
of the 27 lines on a generic cubic surface in characteristic zero.
His findings may be expressed in moduli-theoretic terms, as
follows.

As in \cite[Sec.\ 2.1]{achtercubmod} and \cite[Sec.\ 2]{matsumototerasoma}, a marking on a cubic surface $Y$ is an
identification $\Psi$ of the configuration of lines on $Y$ (i.e., the 27
lines and their incidence relations) with an abstract configuration of
this type.   The set of all markings on a fixed cubic surface is a
torsor under $W(\esix)$, the Weyl group of the exceptional root system
$E_6$.  Let $\stk S^\marked\ra \spec \calo_E[1/6]$ be the moduli space (actually, a fine
moduli scheme) of marked cubic
surfaces.  For a field $K$ equipped with a map $\calo_E[1/6]\ra K$, let $K(\stk
C^\marked)$ be the ring of rational functions on $\stk S^\marked
\cross \spec K$; similarly, let $K(\stk S)$ be the ring of rational
functions on $\coarse{\stk S}\cross \spec K$, where $\coarse{\stk S}$
is the coarse moduli scheme of the stack $\stk S$.

Jordan's results, in modern language, show that $\stk S^\marked_\cx$ is
irreducible, and $\cx(\stk S^\marked)/\cx(\stk S)$ is a Galois
extension of fields with Galois group $W(\esix)$.  In fact, this is true
if $\cx$ is replaced with an arbitrary field, and in particular one of
positive characteristic.

For a prime $\idp \in \calo_E[1/6]$, let $\kappa(\idp)$ be its residue field.

\begin{proposition}
\label{propfieldoflines}
\begin{alphabetize}
\item Each fiber of $\stk S^\marked \ra \spec \calo_E[1/6]$ is
  geometrically irreducible.
\item For $[\idp] \in \spec \calo_E[1/6]$, $\stk S^\marked \cross
  [\idp]$ is geometrically irreducible, and $\kappa(\idp)(\stk
  S^\marked)/\kappa(\idp)(\stk S)$ is a Galois extension of fields
  with group $W(\esix)$.
\end{alphabetize}
\end{proposition}

\begin{proof}
Part (b) is a reformulation of part (a) which, like Lemma
\ref{lemunitarymono}, is a consequence of a good 
theory of arithmetic compactification of PEL Shimura varieties.
Indeed, let $\stk M^{(1-\zeta_3)} \ra \spec \integ[\zeta_3,1/6]$ be the moduli
space of principally polarized abelian $\integ[\zeta_3]$-fivefolds
with full $(1-\zeta_3)$-level structure, i.e., with a 
rigidification of the kernel of multiplication by $1-\zeta_3$; see
\cite[Sec.\ 2.2]{achtercubmod} for more details.  Then $\stk
M^{(1-\zeta_3)} \ra \stk M$ is a $W(\esix)$-torsor.  Moreover, let $\stk
N^{(1-\zeta_3)} = \stk N \cross_{\spec M} \stk M^{(1-\zeta_3)}$;
there is an isomorphism of fine moduli schemes $\stk S^\marked \sriso
\stk N^{(1-\zeta_3)}$ \cite[Thm.\ 5.7]{achtercubmod}.

In particular, for an algebraically closed field $K$ equipped with a morphism
$\calo_E[1/6] \ra K$, there is a canonical bijection between the
irreducible components of $\stk S^\marked_K$ and of $\stk
N^{(1-\zeta_3)}_K$.  Since $\stk N \ra \stk M$ is fiberwise an open,
dense immersion, each of these is in bijection with the irreducible
components of $\stk M^{(1-\zeta_3)}_K$.  Finally, the existence of smooth
compactifications \cite[6.4.1.1]{lanbook} implies that the irreducible
components of $\stk M^{(1-\zeta_3)}_K$ are in bijection with those of
$\stk M^{(1-\zeta_3)}_\cx$.  The desired result now follows from
Jordan's theorem.
\end{proof}

Thus, if $K$ is a Hilbertian field (in which $6$ is invertible), then
for most cubic surfaces $Y/K$, the field of definition $K(\Psi_Y)$ of
the 27 lines on $Y$ is as large as possible; $\gal(K(\Psi_Y)/K) \iso
W(\esix)$.  In particuar, $\gal(K)$ acts as the full automorphism group of the
configuration of 27 lines.

Nonetheless, using the theory of complex multiplication, we are able
to identify a class of cubic surfaces for which $\gal(K)$ acts on
$\Psi_Y$ via an abelian quotient.  For example, suppose $Y \in \stk S(\cx)$
is special, i.e., that the Mumford-Tate group $\mt(X)$ of $X = Q(Y)$ is a torus.
Then $X$ is an abelian variety of CM type, and is defined over a
totally imaginary number field $K$.  The 
Torelli map is an isomorphism of stacks over $E$, and so $Y$ is also
defined over $K$.  Let 
$[Y,\Psi] \in \stk S(\bar K)$ be the moduli point of a marking of the
cubic surface $Y$, and let $[X,\Phi] = \tau^\marked([Y,\Psi])$.  Then
$K(\Psi)\iso K(\Phi)$ is the field obtained by adjoining the
coordinates of all $(1-\zeta_3)$-torsion of $X$ to $K$, and thus is an
abelian extension of $K$.

Moreover, the reciprocity law attached to the Shimura variety $\stk
M^\marked$ \cite{delignecorvallis} explicitly describes the action of
$\gal(K\ab/K)$ on the moduli point $[X,\Phi]$, and thus gives a class-field-theoretic
description of the action of $\gal(K\ab/K)$ on the 27 lines on $Y$.

\bibliographystyle{hamsplain}
\bibliography{jda}

\end{document}